\documentclass[a4paper,12pt]{amsart}
\usepackage[T1]{fontenc}
\usepackage[latin1]{inputenc}
\usepackage[english]{babel}

\usepackage{geometry}

\usepackage{setspace} 
\usepackage{amsfonts}
\usepackage{amssymb}
\usepackage{amsthm}
\usepackage{amsmath}
\usepackage{amscd}
\usepackage{verbatim} 
\usepackage{txfonts}
\usepackage{braket} 
\usepackage{mathtools} 
\usepackage{mathrsfs}
\usepackage{mathabx}

\DeclarePairedDelimiter{\abs}{\lvert}{\rvert}

\theoremstyle{plain}
\newtheorem{theor}{Theorem}[section]
\newtheorem{prop}[theor]{Proposition}
\newtheorem{lemma}[theor]{Lemma}
\newtheorem{cor}[theor]{Corollary}
\newtheorem{fact}[theor]{Fact}

\theoremstyle{definition}
\newtheorem{definition}[theor]{Definition}

\newtheorem{remark}[theor]{Remark}

\renewcommand{\epsilon}{\varepsilon}
\renewcommand{\theta}{\vartheta}
\renewcommand{\rho}{\varrho}
\renewcommand{\phi}{\varphi}

\newcommand{\fleq}[2][2]{\hbox to 0pt{\hss \hbox to\columnwidth{\hspace{#1em} $\displaystyle#2$ \hfil} \hss}}

\newcommand{\Z}{\mathbb{Z}}
\newcommand{\Q}{\mathbb{Q}}

\title{ordered abelian groups that do not have elimination of imaginaries}
\author{martina liccardo\textsuperscript{\textasteriskcentered}}
\thanks{\textsuperscript{\textasteriskcentered}Dipartimento di Matematica e Applicazioni "Renato Caccioppoli", Università degli Studi di Napoli "Federico II", Via Cinthia 26, Monte S. Angelo, Napoli, Italy  \\ martina.liccardo2@unina.it}

\begin{document}

\maketitle

\textsc{abstract.}
{\small We investigate the property of elimination of imaginaries for some special cases of ordered abelian groups. We show that certain Hahn products of ordered abelian groups do not eliminate imaginaries in the pure language of ordered groups. Moreover, we prove that, adding finitely many constants to the language of ordered abelian groups, the theories of the finite lexicographic products $\Z^n$ and $\Z^n \times \Q$ have definable Skolem functions.}

\section*{Introduction}
Proving elimination of imaginaries is essentially a definability problem. Indeed, the intuitive idea for a structure eliminating imaginaries is that some quotient structures, in general not definable, admit canonical codings which essentially define them. Our aim is to investigate this property for some special cases of ordered abelian groups. 

An ordered abelian group $G$ is an abelian group endowed with a linear order compatible with the group operation: $a < a'\implies a+b<a'+b$ for all $a, a', b \in G$. The order is \emph{discrete} if there exists a minimal positive element, \emph{dense} otherwise. Both classes of structures are axiomatizable in $L_{oag} = \Set{0,+,-,<}$ and they play an important role in the model theory of valued fields. Gurevich and Schmitt, from the 60s to the 80s, gave a notable contribution to the model theory of ordered abelian groups, providing, among other things, a test for elementary equivalence of two ordered abelian groups \cite{gurevich}, and a proof that the theory of ordered abelian groups is NIP \cite{NIP}.

In the literature there are partial results on elimination of imaginaries in ordered abelian groups. Examples of ordered abelian groups that eliminate imaginaries are divisible ordered abelian groups $G \equiv (\Q,+,<)$, see \cite{pillay}, and discrete ordered abelian groups $G \equiv (\Z,+,<)$, see \cite{cluckers} and Appendix A of \cite{kovacsics}. One can ask whether every ordered abelian group has elimination of imaginaries in $L_{oag}$. We solve negatively the question producing examples of ordered abelian groups that do not eliminate imaginaries in $L_{oag}$. Indeed, we prove that elimination of imaginaries fails for the theory of any lexicographic products of $\Z$, $\mathscr{H}_{i < \alpha}\Z$, over a well-ordered index set $\alpha>1$, and for the theory of the lexicographic product $\mathscr{H}_{i < \alpha}\Z \times \Q$. 

A property related to the elimination of imaginaries is the existence of definable Skolem functions. In particular, definable Skolem functions allow to reduce the goal of finding a code for every imaginary to coding just one dimensional definable sets. We prove that, once we add finitely many judiciously chosen elements as new constants to $L_{oag}$, the theories of $\Z^n$ and $\Z^n \times \Q$, for any $n\ge1$, have definable Skolem functions.

\section{Notation and preliminaries}
First of all, we recall all definitions and notions that we will use in this paper.

Let $\mathfrak{M}$ be an $L$-structure, $m$ a positive integer and $E(\bar{x},\bar{y})$ a $0$-definable equivalence relation on $\mathfrak{M}^m$. The $E$-equivalence classes of $\mathfrak{M}^m$ are called \emph{imaginary elements} of $\mathfrak{M}$, and we say that $\mathfrak{M}$ eliminates imaginaries if each imaginary can be coded in the structure. More precisely, the notion of elimination of imaginaries is stated in the following definition.

\begin{definition}[\cite{hodges}]
\label{defEI}
An $L$-structure $\mathfrak{M}$ has \emph{elimination of imaginaries} if for any positive integer $m$, any $0$-definable equivalence relation $E(\bar{x},\bar{y})$ on $\mathfrak{M}^m$ and any $E$-class $X$, there is an $L$-formula $\theta(\bar{x},\bar{z})$ such that $X=\theta(\mathfrak{M}^m,\bar{b})$ for some \textsc{unique} tuple $\bar{b} \subseteq \mathfrak{M}$. Moreover, such a tuple $\bar{b}$ is called a \emph{canonical parameter} for $X$.
\end{definition}

Let $\mathfrak{M}$ be an $L$-structure and $S \subseteq \text{Aut}(\mathfrak{M})$. By $\text{Fix}(S)$ we denote the set 
\begin{equation*}
\text{Fix}(S) = \Set{a \in \mathfrak{M} \mid f(a) = a \text{ for all } f \in S}.
\end{equation*}
Let $A \subseteq \mathfrak{M}$. By $\text{Aut}(\mathfrak{M}/A)$ and $\text{Stab}_\mathfrak{M}(A)$ we denote the group of all automorphisms of $\mathfrak{M}$ fixing $A$ pointwise and fixing $A$ setwise, respectively. Namely,
\begin{gather*} 
\text{Aut}(\mathfrak{M}/A)=\Set{ \phi \in \text{Aut}(\mathfrak{M}) \mid \phi(a)=a \text{ for every } a \in A} \\ \text{Stab}_\mathfrak{M}(A)=\Set{\phi \in \text{Aut}(\mathfrak{M}) \mid \phi(A)=A}.
\end{gather*}

\begin{remark}
\label{prEI}
Let $\mathfrak{M}$ be an $L$-structure that admits elimination of imaginaries and let $\bar{b}$ be a canonical parameter for an $E$-class $X$, where $E(\bar{x},\bar{y})$ is a $0$-definable equivalence relation on $\mathfrak{M}^m$. Then $\bar{b}$ is fixed by the same automorphisms of $\mathfrak{M}$ which leave $X$ invariant. Therefore, if $\mathfrak{M}$ eliminates imaginaries, for every $E$-class $X$ there exists a tuple $\bar{b} \subseteq \mathfrak{M}$ such that, for any automorphism $f$ of $\mathfrak{M}$, $f \in \text{Stab}_\mathfrak{M}(X)$ if and only if $f \in \text{Aut}(\mathfrak{M}/\bar{b})$.
\end{remark}

\begin{definition}
\label{uniformEI}
An $L$-structure $\mathfrak{M}$ has \emph{uniform elimination of imaginaries} if for any positive integer $m$, and any $0$-definable equivalence relation $E(\bar{x},\bar{y})$ on $\mathfrak{M}^m$, there is an $L$-formula $\theta(\bar{x},\bar{z})$ such that for every $E$-class $X$, there is a unique tuple $\bar{b} \subseteq \mathfrak{M}$ such that $X=\theta(\mathfrak{M}^m,\bar{b})$.
\end{definition}

In other words, $\mathfrak{M}$ uniformly eliminates imaginaries if one can find a formula $\theta(\bar{x},\bar{z})$ as in Definition \ref{defEI} depending only on the equivalence relation $E(\bar{x},\bar{y})$ and not on the equivalence class $X$. Notice that uniform elimination of imaginaries is preserved under elementarily equivalence.

We say that a theory $T$ in $L$ has (uniform) elimination of imaginaries if every model of $T$ has (uniform) elimination of imaginaries. The uniformity of elimination of imaginaries, in the sense of Definition \ref{uniformEI}, holds under the condition stated in the following well-known result (see, for example, \cite[Theorem 16.16]{poizatbook}).

\begin{theor}
Let $T$ be an $L$-theory. Suppose $T$ has elimination of imaginaries and, for every model of $T$, $dcl(\emptyset)$ contains at least two elements. Then $T$ has uniform elimination of imaginaries.
\end{theor}
Note that the existence of at least two $0$-definable elements in order to have uniform elimination of imaginaries cannot be avoided. Indeed, both classes of the equivalence relation $(x_1=x_2 \land y_1 = y_2) \lor (x_1 \ne x_2 \land y_1 \ne y_2)$ are definable without parameters by the formulas $x_1 = x_2$ and $x_1 \ne x_2$, respectively, but there is no way of using a unique tuple to pick out one of these formulas.

Another form of elimination of imaginaries is the following.
\begin{definition}
We say that an $L$-structure $\mathfrak{M}$  has \emph{weak elimination of imaginaries} if for any positive integer $m$, any $0$-definable equivalence relation $E(\bar{x},\bar{y})$ on $\mathfrak{M}^m$ and any $E$-class $X$, there are a $L$-formula $\theta(\bar{x},\bar{z})$ and a \textsc{finite} set $B$ of tuples of $\mathfrak{M}$ such that $X=\theta(\mathfrak{M}^m,\bar{b})$ if and only if $\bar{b} \in B$. \\ We say that a theory $T$ in $L$ has weak elimination of imaginaries if every model of $T$ has weak elimination of imaginaries.
\end{definition}

\begin{fact}[\cite{poizat}]
\label{factpoizat}
Let $T$ be an expansion of the theory of linear order. If $T$ weakly eliminates imaginaries, then $T$ eliminates imaginaries.
\end{fact}

Throughout the paper, unless otherwise stated, we work in the language of ordered abelian groups $L_{oag}$, and "definable" will mean definable with parameters.

Let $G$ be an ordered abelian group, and $f,g \in G$ such that $f<g$. The notations $[f,g]$ and $(f,g)$ stand for the intervals $\{h \in G \mid f \le h \le g\}$ and $\{h \in G \mid f < h <g\}$, respectively. The absolute value of $f$ is $\abs{f}= \max\{f,-f\}$.
Let $m \in \mathbb{N}$ be a positive integer, $\equiv_m$ denotes the binary relation on $G$ defined by 
\begin{equation*}
f \equiv_m g \text{ if and only if } f-g \in mG, 
\end{equation*}
Clearly, the relation $\equiv_m$ is a $0$-definable equivalence relation.

A subgroup $\Delta$ of an ordered abelian group $G$ is \emph{convex} in $G$ if, for every $g,h \in G$ such that $0 < g < h$ and $h \in \Delta$, we also have $g \in \Delta$. The set of all convex subgroups of $G$ is linearly ordered by inclusion and is denoted by $\mathcal{CS}(G)$. The order type of the collection of all proper convex subgroups of $G$ is called the \emph{rank} of $G$. We now introduce the following two families of definable convex subgroups, $\{H_n(f)\}_{f \in G \setminus nG}$ and $\{H_n^+(f)\}_{f \in G}$, see \cite{immi} for details. For any $n \in \mathbb{N}$ and $f \in G \setminus nG$, $H_n(f)$ denotes the largest convex subgroup $H$ of $G$ not intersecting $f + nG$. In particular,
\begin{equation*}
H_n(f)=\Set{h \in G \mid \Braket{h}^{conv} \cap f+nG = \emptyset},
\end{equation*}
where $\Braket{h}^{conv}$ is the smallest convex subgroup of $G$ containing $h$. For every $f \in G$, $H_n^+(f)$ denotes the definable convex subgroup
\begin{equation*}
H_n^+(f)=\underset{\Set{g \in G \setminus nG \mid f \in H_n(g)}}{\bigcap}H_n(g).
\end{equation*}

Let $(I,<)$ be an ordered set, and for each $i \in I$ let $G_i$ be an ordered abelian group. Consider the direct product of the groups $G_i$, $\prod_{i \in I}G_i$. For every $f \in \prod_{i \in I}G_i$, the \emph{support} of $f$ is the set
\begin{equation*}
supp(f)=\{i \in I \mid f(i) \ne 0\}.
\end{equation*}
Let
\begin{equation*}
H=\Set{f \in \prod_{i \in I}G_i \mid supp(f) \text{ is a well-ordered subset of } (I,<)}.
\end{equation*}
It is easily checked that $H$ is a subgroup of $\prod_{i \in I}G_i$ and is an ordered abelian group by defining
\begin{equation*}
f<g \text{ if and only if } f(i) < g(i), \text{ where } i=\min supp(g-f),
\end{equation*}
for any $f,g \in H$. The group $H$ endowed with this order is the \emph{Hahn product} of the family $\{G_i\}_{i \in I}$ and is denoted by $\mathscr{H}_{i \in I} G_i$. The ordered subgroup 
\begin{equation*}
\Set{ f \in \mathscr{H}_{i \in I} G_i \mid supp(f) \text{ is finite }}
\end{equation*}
is the \emph{lexicographic sum} of the family $\{G_i\}_{i \in I}$ and is denoted by $\sum_{i \in I}G_i$.

The lexicographic sum and the Hahn product of a family of ordered abelian groups are indistinguishable by first order properties. Indeed, the following holds.

\begin{prop}[\cite{habschmitt}, Corollary 6.3]
\label{elemsub}
Let $(I,<)$ be an ordered set and for each $i \in I$ let $G_i$ be an ordered abelian group. The lexicographic sum $\sum_{i\in I}G_i$ is an elementary substructure of the Hahn product $\mathscr{H}_{i \in I}G_i$.
\end{prop}

We denote the successor of an ordinal $\alpha$ by $\alpha+1$ and, if a set $S$ has order type $\alpha$, the order type of the reverse order of $S$ is denoted by $\alpha^{*}$.

\section{Lexicographic products of $\Z$}

In this section, we consider the Hahn product of the group of integers, $\Z$, with the usual order over a well-ordered set $I$ with $\abs{I}>1$. Henceforth, we identify $I$ with the corresponding ordinal $\alpha$, and $\Lambda$ and $\Gamma$ will denote $\mathscr{H}_{i < \alpha}\Z$ and $\sum_{i<\alpha}\Z$, respectively. Note that, since $\alpha$ is well-ordered, the domain of $\Lambda$ coincides with the direct product $\prod_{i < \alpha}\Z$. This allows us to refer to the Hahn product $\mathscr{H}_{i < \alpha}\Z$ as the lexicographic product of $\Z$ over $\alpha$. We assume that $\alpha>1$ since the case $\alpha=1$ is already studied in \cite{cluckers}.

If $f \in \Lambda$, $f \neq 0$, then we define $v(f)=\min supp(f)$. Furthermore, we add the symbol $\infty$ to $\alpha$ and put $i<\infty$ for all $i < \alpha$. Hence, we set $v(0)=\infty$. One can easily prove the following properties.
\begin{fact}
\label{propv}
(i) $v(f)=v(-f)$ for any $f \in \Lambda$; \\
(ii) $v(f+g) \ge \min\{v(f),v(g)\}$ for any $f,g \in \Lambda$; 
\\
(iii) Let $f, g \in \Lambda$ such that $0<g<f$. Then $v(f) \le v(g)$.
\end{fact}

Let $i < \alpha$, $e_{\{i\}}$ will denote the following element of $\Lambda$:
\begin{equation}
\label{basisGamma}
(e_{\{i\}}(j))_{j < \alpha} \text{ where } e_{\{i\}}(j) =
\begin{cases}
1 \text{ if } j=i \\
0 \text{ otherwise}
\end{cases}
\end{equation}
Note that, for every $i,j < \alpha$ such that $i<j$, $0<e_{\{j\}} < e_{\{i\}}$. Moreover, the family $\{e_{\{i\}}\}_{i < \alpha}$ generates $\Gamma=\sum_{i<\alpha}\Z$ as an abelian group.

\begin{prop}
\label{order}
If $\alpha=\beta+1$ for some ordinal $\beta$, then $\Lambda=\mathscr{H}_{i<\alpha}\Z$ is discrete and, in particular, $e_{\{\beta\}}$ is the minimal positive element of $\Lambda$. Otherwise, if $\alpha$ is a limit ordinal, $\Lambda$ is dense.
\end{prop}
\begin{proof}
Let $\alpha=\beta+1$ and suppose $f \in \Lambda$ is such that $0 < f \le e_{\{\beta\}}$. Then $v(f)=\beta$ and $0<f(\beta) \le 1$. Since $f(\beta) \in \Z$, $f(\beta)=1$ and $f=e_{\{\beta\}}$. \\
Let $\alpha$ be a limit ordinal and $f \in \Lambda$, $f>0$. Let $i=v(f)$. Then $f(i)>0$ and $f(j)=0$ for every $j<i$. Clearly, $0<e_{\{i+1\}}<f$.
\end{proof}
Note that, since $\sum_{i<\alpha}\Z \preceq \mathscr{H}_{i<\alpha}\Z$, if $\alpha=\beta+1$, $e_{\{\beta\}}$ is the minimal positive element of $\Gamma=\sum_{i<\alpha}\Z$. Otherwise, $\sum_{i<\alpha}\Z$ is dense.

Now we aim at characterizing the convex subgroups of $\Lambda$. To this purpose, for every $i \in \alpha \cup \{\infty\}$, define $\Lambda_i=\Set{f \in \Lambda \mid v(f) \ge i}$. Trivially, $\Lambda_0=\Lambda$, $\Lambda_{\infty}=\{0\}$ and, for every $0<i<\infty$, $\Lambda_i$ is a (non-trivial) convex subgroup of $\Lambda$. Let
\begin{equation*}
F \colon i \in \alpha \cup \{\infty\} \mapsto \Lambda_i \in \mathcal{CS}(\Lambda)
\end{equation*}
It is clear that $F$ is injective and order-reversing: $i \le j$ implies $\Lambda_j \subseteq \Lambda_i$, for every $i,j \in \alpha \cup \{\infty\}$. The following proposition implies that $F$ is also surjective.

We recall that an endsegment of $\alpha \cup \{\infty\}$ is a subset of $\alpha \cup \{\infty\}$ that is closed upward. Let $\mathcal{ES}(\alpha \cup \{\infty\})$ denote the set of all endsegments of $\alpha \cup \{\infty\}$.

\begin{prop}
\label{convexsub}
There exists an order-preserving bijection between the set of all endsegments of $\alpha \cup \{\infty\}$ and the set of all convex subgroups of $\Lambda=\mathscr{H}_{i<\alpha}\Z$.
\end{prop}

\begin{proof}
Let $C$ be an endsegment of $\alpha \cup \{\infty\}$ and $\Delta_{C}= \Set{f \in \Lambda \mid v(f) \in C}$. By Fact \ref{propv}, $\Delta_C$ is  a convex subgroup of $\Lambda$. Consider the function
\begin{equation*}
K \colon C \in \mathcal{ES}(\alpha \cup \{\infty\}) \mapsto \Delta_{C} \in \mathcal{CS}(\Lambda).
\end{equation*}
Trivially, if $C_1, C_2 \in \mathcal{ES}(\alpha \cup \{\infty\})$ are such that $C_1 \subseteq C_2$, then $\Delta_{C_1} \subseteq \Delta_{C_2}$. So $K$ is order-preserving.

Let $\Delta$ be a convex subgroup of $\Lambda$. Then $C_{\Delta}= \Set{v(f) \in \alpha \cup \{\infty\} \mid f \in \Delta}$ is an endsegment of $\alpha \cup \{\infty\}$. Indeed, $\infty \in C_{\Delta}$, and let $k < \alpha$ with $j < k$ for some $j \in C_{\Delta}$. Then there exists $f \in \Delta$ such that $j=v(f)$ and, without loss of generality, we can suppose $f>0$. Clearly, $0<e_{\{k\}}<f$ and $v(e_{\{k\}})=k$. By convexity of $\Delta$, $e_{\{k\}} \in \Delta$ and $k \in C_{\Delta}$. Therefore, the function
\begin{equation*}
H \colon \Delta \in \mathcal{CS}(\Lambda) \mapsto C_{\Delta} \in \mathcal{ES}(\alpha \cup \{\infty\}).
\end{equation*} 
is order-preserving. Moreover, $H(K(C))=C$ and $K(H(\Delta))=\Delta$ for every $C \in \mathcal{ES}(\alpha \cup \{\infty\})$ and $\Delta \in \mathcal{CS}(\Lambda)$. Then, $H=K^{-1}$, and so $K$ is bijective.
\end{proof}

Note that any endsegment $C$ of $\alpha \cup \{\infty\}$ is either $\{\infty\}$ or $C=\Set{j < \alpha \mid j \ge i} \cup \{\infty\}$ for some $i < \alpha$. This implies that the $\Lambda_i$'s are exactly all the convex subgroups of $\Lambda$.

By Proposition \ref{convexsub}, we can determine the order type of the set of all proper convex subgroups of $\Lambda$. 

\begin{cor}
\label{corconvexsub}
For every ordinal $\alpha>1$, $\Lambda=\mathscr{H}_{i<\alpha}\Z$ is of rank $(\alpha+1)^*$. \\ In particular, if $\alpha$ is finite, i.e. $\alpha=n$ for some $n \in \mathbb{N}$, $n>1$, there are exactly $n$ proper convex subgroups of $\Lambda$.
\end{cor}

Note that, since
\begin{equation*}
\mathcal{CS}(\Gamma)=\Set{ \Delta \cap \Gamma \mid \Delta \in \mathcal{CS}(\Lambda)},
\end{equation*}
the sets $\Gamma_i=\Set{f \in \Gamma \mid v(f) \ge i}$  are exactly all the convex subgroups of $\Gamma$, and Proposition \ref{convexsub} and Corollary \ref{corconvexsub} hold also for $\Gamma=\sum_{i<\alpha}\Z$. 

\begin{prop}
\label{defconvexsub}
The family $\{\Lambda_{i+1}\}_{i+1<\alpha}$ is uniformly definable.  
\end{prop}
\begin{proof}
Let $i < \alpha$ be such that $i+1<\alpha$. Then $\Lambda_{i+1}=H_2(e_{\{i\}})$ and, so, is definable. Indeed, for every $f \in e_{\{i\}}+2\Lambda$, we have $v(f) \le i$. Hence, $H_2(e_{\{i\}})=\Set{f \in \Lambda \mid v(f) > i} = \Set{f \in \Lambda \mid v(f) \ge i+1}=\Lambda_{i+1}$.
\end{proof}

More generally, for every $f \in \Lambda \setminus n\Lambda$, we have
\begin{equation*}
H_n(f)=
\begin{cases}
\Lambda_{i+1} & \text{if $i+1<\alpha$} \\
\{0\} & \text{otherwise} 
\end{cases}
\end{equation*}
where $i=\min \Set{j < \alpha \mid f(j) \notin n\Z}$.

Now we are able to prove 
\begin{cor}
All convex subgroups of $\Lambda$ are definable. 
\end{cor}
\begin{proof}
We have only to show the definability of $\Lambda_i$ for $i$ limit ordinal. Let $i<\alpha$ be a limit ordinal, then
\begin{equation*}
\Lambda_i=\underset{\Set{j<\alpha\mid j<i}}{\bigcap}\Lambda_j=\underset{\Set{j<\alpha\mid j<i \text{ and $j$ is not limit }}}{\bigcap}\Lambda_j=H_2^+(e_{\{i\}}).
\end{equation*}
\end{proof}
It is clear that the above statements hold also for $\Gamma=\sum_{i<\alpha}\Z$.

\subsection{The group of automorphisms of the lexicographic sum of $\Z$}

In this subsection, we will describe the group of automorphisms of $\Gamma=\sum_{i<\alpha}\Z$. 

First of all, we characterize the automorphisms of $\Gamma$ when $\alpha$ is finite, i.e. $\alpha=n$ for some natural number $n>1$. In this case, $\Gamma=\sum_{0 \le i \le n-1}\Z = \mathscr{H}_{0 \le i \le n-1}\Z$, and we will denote $\Gamma$ by its domain $\Z^n$. For any $z \in \Z^n$, let $z(i)=z_i$ for every $0 \le i \le n-1$. Note that $\Z^n$ is a discrete ordered abelian group. Therefore, every element of the convex subgroup 
\begin{equation*}
{\Z^n}_{n-1} = \Set{ z \in \Z^n \mid v(z) \ge n-1} = \Set{(0,\dots, 0, m) \mid m \in \Z}
\end{equation*}
is $0$-definable. Moreover, each convex subgroup ${\Z^n}_i$ of $\Z^n$ is $0$-definable, as it is proved in the following lemma.

\begin{lemma}
	\label{defconvex}
	For every $0 \le i \le n-1$, ${\Z^n}_i=\Set{z \in \Z^n \mid v(z) \ge i}$ is $0$-definable.
\end{lemma}
\begin{proof}
For $i=0$ it is trivial. So, let $0 < i \le n-1$ and fix a prime $p$. For every $z \in \Z^n$, $z \in {\Z^n}_i$ if and only if the set $\Set{w+p\Z^n \mid -z \le w \le z}$ of the $\equiv_p$ - equivalence classes of elements in the interval $[-z,z]$ has cardinality at most $p^{n-i}$. Indeed, let $z \in \Z^n_i$. Then $[-z,z] \subseteq (-e_{\{i-1\}},e_{\{i-1\}})$. Since, for every $0 \le j \le n-1$ the $\equiv_p$ - equivalence classes in $(-e_{\{j\}},e_{\{j\}})$ are exactly $p^{n-1-j}$, it follows that $\abs{\Set{w+p\Z^n \mid -z \le w \le z}} \le p^{n-i}$. Conversely, if $z \notin \Z^n_i$, then $v(z)<i$ and $[-z,z]$ contains the interval $[-e_{\{i-1\}},e_{\{i-1\}}]$. Therefore $|\Set{w+p\Z^n \mid -z \le w \le z}| \ge p^{n-i}+1$.

Clearly, the set of $z \in \Z^n$ such that $\abs{\Set{w+p\Z^n \mid -z \le w \le z}} \le p^{n-i}$ is $0$-definable.
\end{proof}

Consider the following upper triangular matrix of size $n \times n$, whose elements are in $\Z$
\begin{equation}
\label{triangmatrix}
K_n =
\begin{pmatrix}
1  & k_{12} & \dots & k_{1n} \\
0 & 1 & \ddots & \vdots \\
\vdots & \ddots & \ddots & k_{n-1n}\\
0 & \dots & 0 & 1
\end{pmatrix}
.
\end{equation}

Note that $K_n$ is invertible and its inverse is an upper triangular matrix on $\Z$ of the same form, with all the entries of the main diagonal equal to $1$. Then the function 
\begin{equation*}
f_{K_n} \colon z \in \Z^n \mapsto zK_n \in \Z^n
\end{equation*}
where $zK_n$ is the matrix product of the row vector $z=(z_1,\dots,z_n)$ and $K_n$, is a group automorphism of $\mathbb{Z}^n$. Moreover, it is straightforward to show that $f_{K_n}$ is order-preserving. Therefore, every matrix as in (\ref{triangmatrix}) determines an automorphism of the ordered group $\Z^n$. We now prove that every automorphism of $\mathbb{Z}^n$ is obtained in this way and it can be represented by a matrix as in (\ref{triangmatrix}).

Note that, since $\Z^n$ is a free module of rank $n$, every group automorphism $\phi$ of $\Z^n$ is given by an invertible matrix $M \in GL_n(\Z)$. In other words, there exists an $n \times n$ matrix $M$ over $\Z$ with $\text{det}(M) = \pm 1$ such that $\phi(z)=zM$ for every $z \in \Z^n$. Indeed, the matrix
\begin{equation}
\label{automatrix}
M =
\begin{pmatrix}
\phi(e_{\{1\}}) \\
\vdots \\
\phi(e_{\{n\}})
\end{pmatrix}
\end{equation}
represents $\phi$ in the above sense. Hence, it suffices to prove that if $\phi$ preserves the order on $\Z^n$, then $M$ is of the form (\ref{triangmatrix}). 

Let $0 \le i \le n-1$ and $m$ be any element in $\Z \setminus \{0\}$, consider 
\begin{equation*}
A_{i,m}= \Set{ z \in \Z^n \mid v(z)=i \text{ and } z_{v(z)}=m}.
\end{equation*}
Note that, for $i=n-1$, $A_{n-1,m}$ is the singleton of $(0,\dots,0,m)$. For $i=0$, $A_{0,m}$ is the set $\{m\} \times \Z^{n-1}$, and, for $0<i<n-1$, $A_{i,m}= \{0\}^i \times \{m\} \times \Z^{n-(i+1)}$.

\begin{prop}
\label{defin}
For every $0 \le i \le n-1$ and every $m \in \Z \setminus \{0\}$, the set $A_{i,m}$ is $0$-definable.
\end{prop}
\begin{proof}
For $i=n-1$ it is clear. Let $0 \le i < n-1$ be fixed and consider the set $A_i=\Set{z \in \Z^n \mid v(z)=i}$. By Lemma \ref{defconvex}, since $A_i = {\Z^n}_i \setminus {\Z^n}_{i+1}$, $A_i$ is $0$-definable. Let $\alpha(x)$ and $\beta(x)$ be the formulas defining $A_i$ and ${\Z^n}_{i+1}$, respectively. Since $A_{i,m} = \Set{z \in \Z^n \mid -z \in A_{i,-m}}$, it suffices to prove the statement for $m>0$. If $m=1$, then $A_{i,1}$ is defined by the formula
\begin{equation*}
	\psi_1(x) \coloneq \alpha(x) \land 0<x \land \forall y \big ( (\alpha(y) \land 0<y \land y<x) \implies \beta(y-x) \big ).
\end{equation*}
Indeed, it is trivial that any element $z$ in $A_{i,1}$ satisfies $\psi_1(x)$. Conversely, let $x$ be an element of $\Z^n$ satisfying $\psi_1(x)$. Then $v(x)=i$ and $x_{v(x)} >0$. If $x_{v(x)} >1$, then $v(x-e_{\{i\}})=i$. Since $0<e_{\{i\}}<x$, we get a contradiction and $x_{v(x)}=1$. Let $m>1$ and suppose by induction that $A_{i,m-1}$ is defined by the formula $\psi_{m-1}(x)$. Consider the formula
\begin{equation*}
\gamma_{m-1}(x) \coloneq \forall y(\psi_{m-1}(y) \implies y<x).
\end{equation*}
Then, $A_{i,m}$ is defined by the formula
\begin{equation*}
\psi_m(x) \coloneq \alpha(x) \land \gamma_{m-1}(x) \land \forall y \big ( \gamma_{m-1}(y) \land y<x) \implies \beta(y-x) \big ).
\end{equation*}
Indeed, let $x$ in $\Z^n$ satisfy $\psi_m(x)$. Then $v(x)=i$ and $x_{v(x)} > m-1$. If $x_{v(x)} > m$, then $v(x-me_{\{i\}})=i$, and we get a contradiction. Therefore, $x_{v(x)}=m$ and $x \in A_{i,m}$.
\end{proof}

Now we are able to prove the following
\begin{theor}
	\label{carattautofinite}
	Let $\phi$ be an automorphism of the lexicographic sum $\mathbb{Z}^n$, where $n>1$. Then there exists an upper triangular matrix $K_n$ as in (\ref{triangmatrix}) such that $\phi = f_{K_n}$.
\end{theor}
\begin{proof}
For every $z \in \Z^n$, $\phi(z)=zM$, where $M$ is the matrix (\ref{automatrix}). By Proposition \ref{defin}, $\phi$ fixes $A_{i,1}$ setwise for every $0 \le i \le n-1$. Therefore, for any $0 \le i \le n-1$,  $v(\phi(e_{\{i\}}))=i$ and ${\phi(e_{\{i\}})}_i=1$, so the result is proved.
\end{proof}

Henceforth we focus on the case $\alpha \ge \omega$. Recall that the chain
\begin{equation}
\{0\} = \Gamma_{\infty} \subset \dots \subset \Gamma_n \subset \dots \subset \Gamma_1 \subset \Gamma_{0} = \Gamma = \sum_{i<\alpha}\Z
\end{equation}
represents the set of all convex subgroups of $\Gamma$. Note that, contrary to the finite case, it is not a well-ordered set, since it contains an infinite descending chain. 

Let $\phi$ be an automorphism of $\Gamma$. For every $f \in \Gamma$, there exists a finite subset $\Set{i_1,\dots,i_m}$ of $\alpha$ such that $f=\sum_{j=1}^{m}f(i_j)e_{\{i_j\}}$. Then, since $\phi$ is a group homomorphism, $\phi(f)=\sum_{j=1}^{m}f(i_j)\phi(e_{\{i_j\}})$. For every $i,k < \alpha$, we set
\begin{equation*}
m_{ik}=(\phi(e_{\{i\}}))(k).
\end{equation*}
Then, we have
\begin{equation*}
(\phi(f))(k)=\sum_{i< \alpha}f(i)m_{ik}
\end{equation*}
for every $k<\alpha$. In other words, also in the case of $\alpha$ infinite, we can associate to $\phi$ an infinite matrix $M=(m_{ik})_{i,k < \alpha}$ such that $\phi(f)=fM$ for every $f \in \Gamma$. Note that any row of $M$ has only finitely many nonzero elements.

In order to generalize the result of Theorem \ref{carattautofinite}, we now determine which conditions $M$ has to satisfy, for preserving the order on $\Gamma$. We will prove that also in the infinite case $M$ is upper triangular with all the entries of the main diagonal equal to $1$, i.e.
\begin{equation}
\label{triangmatrixinfinite}
m_{ik}=
\begin{cases}
0 \text{ if } i>k \\
1 \text{ if } i=k
\end{cases}
\end{equation}
 
\begin{lemma}
\label{fixingGammai}
Let $\phi$ be an automorphism of $\Gamma$. Then $\phi$ fixes $\Gamma_i$ setwise for every $0 \le i < \alpha$.
\end{lemma}
\begin{proof}
Note that, for every $0 \le i < \alpha$, $\phi(\Gamma_i)$ is a convex subgroup of $\Gamma$. It follows that $\phi$ induces the order-preserving bijection $\Gamma_i \in \mathcal{CS}(\Gamma) \mapsto \phi(\Gamma_i) \in \mathcal{CS}(\Gamma)$, which will be denoted again by $\phi$. We prove that the automorphism induced by $\phi$ on $\mathcal{CS}(\Gamma)$ is the identity. Clearly, $\phi(\Gamma_0)=\Gamma_0$. Let $0<i<\alpha$ and suppose, by induction, $\phi(\Gamma_j)=\Gamma_j$ for every $j<i$. Therefore, $\phi(\Gamma_i) \subseteq \Gamma_i$. Let $\Delta \in \mathcal{CS}(\Gamma)$ be such that $\phi(\Delta)=\Gamma_i$. Then, $\Delta=\Gamma_h$ for some $h \ge i$. It follows that $\Gamma_h \subseteq \Gamma_i$ and, since $\phi$ is order-preserving, $\Gamma_i \subseteq \phi(\Gamma_i)$. Therefore, $\phi(\Gamma_i)=\Gamma_i$, and the statement is proved.
\end{proof}
From Lemma \ref{fixingGammai}, it follows immediately that
\begin{cor}
\label{corfixingGammai}
If $\phi$ is an automorphism of $\Gamma$, then $\phi$ fixes $\Set{f \in \Gamma \mid v(f)=i}$ setwise for every $0 \le i < \alpha$. 
\end{cor}
\begin{proof}
Let $i<\alpha$ be such that $i+1<\alpha$. Then $\Set{f \in \Gamma \mid v(f)=i}=\Gamma_i \setminus \Gamma_{i+1}$ and, by Lemma \ref{fixingGammai}, it is fixed setwise. Furthermore, if $\alpha=\beta+1$ for some $\beta<\alpha$, each element of $\Gamma_{\beta}$ is $0$-definable and  $\Gamma_{\beta}$ is fixed pointwise.
\end{proof}

\begin{theor}
\label{carattautoinfinite}
Let $\phi$ be an automorphism of $\Gamma$. Then there exists $M=(m_{ik})_{i,k <\alpha}$ with $m_{ik}$'s as in (\ref{triangmatrixinfinite}) such that $\phi(f)=fM$ for every $f \in \Gamma$.
\end{theor}
\begin{proof}
Fix $i< \alpha$ and consider $e_{\{i\}} \in \Gamma$. Set $m_{ik}=(\phi(e_{\{i\}}))(k)$ for every $k<\alpha$. Then by Corollary \ref{corfixingGammai}, $m_{ik}=0$ for $k<i$ and $m_{ik} \ge 1$ for $k=i$. Note that Lemma \ref{fixingGammai} implies that each $\Gamma_i=\Set{f \in \Gamma \mid v(f) \ge i}$ is generated  by $\{\phi(e_{\{j\}})\}_{j\ge i}$ as an abelian group. Therefore, for every $g \in \Gamma_i$ with $v(g)=i$ we have $g(i)=km_{ii}$ for some $k \in \Z$. Therefore $g(i) \equiv 0$ $(\mod m_{ii})$ for every $g \in \Gamma_i$. Then $m_{ii}=1$, and so the statement is proved. 
\end{proof}

Summarizing, we have shown that, for any ordinal $\alpha$, $\alpha >1$, every automorphism of $\Gamma=\sum_{i< \alpha}\Z$ can be represented as a matrix $M=(m_{ik})_{i,k <\alpha}$ with $m_{ik}$'s as in (\ref{triangmatrixinfinite}).

\section{Failure of elimination of imaginaries}

We now prove that in both cases $\Gamma=\sum_{i<\alpha}\Z$ and $\Lambda=\mathscr{H}_{i<\alpha}\Z$, there exist some imaginaries of the ordered abelian group that can not be eliminated. We will first focus on the case of the lexicographic sum.

\begin{theor}
\label{NOTEI}
Let $\alpha>1$ be an ordinal. Then $\Gamma=\sum_{i< \alpha} \Z$ does not admit elimination of imaginaries in the pure language of ordered abelian groups.
\end{theor}

\begin{proof}
Suppose by contradiction that $\Gamma$ admits elimination of imaginaries. 

Let $p$ be a prime and fix an element $a \in \Gamma$ such that $a(0) \notin p\Z$. Let $X=[a]_{\equiv_p}$ be the $\equiv_p$ - equivalence class of $a$. Since $\equiv_p$ is $0$-definable, there exists a canonical parameter $\bar{b}$ for $X$, let $\bar{b}=(b_j)_{j< \mu}$ for some positive integer $\mu$.

Let $\phi$ be an automorphism of $\Gamma$ and $M$ the matrix $(m_{ik})_{i,k < \alpha}$ such that $\phi(f)=fM$ for every $f \in \Gamma$. Recall that $M$ is either finite or infinite depending on the cardinality of $\alpha$, and that, for every $i,k< \alpha$, $m_{ii}=1$ and $m_{ik}=0$ for $k<i$. Suppose $\phi \in \text{Stab}_{\Gamma}(X)$. Then $\phi(a) \equiv_p a$, namely, $(\phi(a))(i) - a(i) \in p\Z$ for every $i < \alpha$. Therefore, from $(\phi(a))(1)=a(0)m_{01}+a(1)$ and $a(0) \notin p\Z$, it follows that $m_{01} \in p\Z$. Moreover, since $\phi \in \text{Stab}_{\Gamma}(X)$ then, by Fact \ref{prEI}, $\phi \in \text{Aut}(\Gamma / \bar{b})$, namely, $\phi(b_j)=b_j$ for every $j<\mu$. In particular, since for every $0<k <\alpha$, $(\phi(b_j))(k)=b_j(k)+\sum_{0 \le i < k}(b_j(i))m_{ik}$ for every $j < \mu$, it follows that 
\begin{equation}
\label{fixingb}
\sum_{0 \le i <k}(b_j(i))m_{ik}=0 \text{ for every } 0<k<\alpha \text{ and for every } j < \mu.
\end{equation}

Let $h<\alpha$ be such that $h+1 < \alpha$, and consider the following matrix $\tilde{M}^h=(\tilde{m}_{ik}^h)_{i,k < \alpha}$ where
\begin{equation*}
\tilde{m}_{ik}^h=
\begin{cases}
1 & \text{ if } i=k\\
p & \text{if } i=h,k=h+1 \\
0 & \text{otherwise}
\end{cases}
\end{equation*}
In particular, $\tilde{M}^h$ is an upper triangular matrix, with all entries of the diagonal equal to $1$. Since for every $f \in \Gamma$, $(f\tilde{M}^h)(h+1)=p(f(h))+f(h+1)$ and $(f\tilde{M}^h)(k)=f(k)$ for every $k<\alpha$, $k \ne h$, $f\tilde{M}^h \in \Gamma$ and $\tilde{M}^h$ induces the function
\begin{equation*}
\phi_{\tilde{M}} \colon f \in \Gamma \mapsto f\tilde{M}^h \in \Gamma.
\end{equation*}
In particular, $\phi_{\tilde{M}^h}(f) \equiv_p f$ for every $f \in \Gamma$. Moreover, $\phi_{\tilde{M}^h}$ is a group automorphism and is order-preserving, hence $\phi_{\tilde{M}^h} \in \text{Stab}_{\Gamma}(X)$. From (\ref{fixingb}) it follows that $p(b_j(h))=0$ for every $j<\mu$, and, so, $b_j(h)=0$ for every $j<\mu$. Now we need to distinguish two cases, $\alpha$ limit ordinal and $\alpha$ successor ordinal.
\begin{description}
\item[case $\alpha$ limit] Since for every $h<\alpha$, $h+1<\alpha$, we obtain $b_j=0$ for every $j<\mu$.
\item[case $\alpha=\beta+1$ for some $\beta<\alpha$] Then $b_j(h)=0$ for every $h<\beta$ and for every $j<\mu$. In particular, for every $j<\mu$, $b_j \in \Gamma_{\beta}$, i.e. $b_j=k_je_{\{\beta\}}$ for some $k_j \in \Z$. Since $e_{\{\beta\}}$ is the minimal positive element of $\Gamma$, $b_j$ is $0$-definable for every $j<\mu$.
\end{description}
Therefore, in both cases, all the automorphisms of $\Gamma$ fix $\bar{b}$. Hence, by Fact \ref{prEI}, any automorphism of $\Gamma$ fixes $X$. This is clearly false. Indeed, for example, the automorphism $\psi \colon \Gamma \to \Gamma$ defined by $\psi(f)=f\bar{M}$, for every $f \in \Gamma$, and $\bar{M}=(\bar{m}_{ik})_{i,k < \alpha}$ with
\begin{equation*}
\bar{m}_{ik}=
\begin{cases}
1 & \text{ if either } i=k \text{ or } i=0,k=1 \\
0 & \text{otherwise}
\end{cases} 
\end{equation*}
does not fix $X$ since $\bar{m}_{01} \notin p\Z$. The contradiction follows from the existence of a canonical parameter for $X$, and so $\Gamma$ does not admit elimination of imaginaries.
\end{proof}

Using the same argument as in Theorem \ref{NOTEI}, we can prove the failure of elimination of imaginaries in $L_{oag}$ also for $\Lambda=\mathscr{H}_{i<\alpha}\Z$. Indeed, let $p$ be a prime, $a \in \Lambda$ be such that $a(0) \notin p\Z$ and $X=[a]_{\equiv_p}$. Fix $h < \alpha$, then the automorphism $\phi_h \colon \Lambda \to \Lambda$ defined by
\begin{equation}
\label{phih}
(\phi_h(f))(j)=
\begin{cases}
pf(j-1)+f(j) & \text{ if } j=h+1 \\
f(j) & \text{otherwise}
\end{cases}
\end{equation}
for every $f \in \Lambda$, fixes $X$ setwise. Therefore, if $\bar{b}$ is a canonical parameter for $X$, we obtain that $\bar{b}$ is a tuple of $ke_{\{\beta\}}$'s, with $k \in \Z$, if $\alpha=\beta+1$ for some $\beta<\alpha$, and $\bar{b}=\bar{0}$ otherwise. In both cases  we have a contradiction. Hence, we have proved
\begin{theor}
\label{NOTEIHAHN}
Let $\alpha>1$ be an ordinal. Then $\Lambda=\mathscr{H}_{i< \alpha} \Z$ does not admit elimination of imaginaries in the pure language of ordered abelian groups.
\end{theor}

Summarizing, the argument used in \ref{NOTEI} and \ref{NOTEIHAHN} consists in determining, for a fixed $E$-equivalence class $X$ of $\Gamma$, where $E$ is a $0$-definable equivalence relation,
\begin{enumerate}
	\item a set $S \subseteq \text{Stab}_{\Gamma}(X)$ with $\text{Fix}(S) \subseteq \text{Fix(Aut}(\Gamma))$
	\item an automorphism $\phi \in \text{Aut}(\Gamma) \setminus \text{Stab}(X)$.
\end{enumerate}
This argument can be adapted for proving the failure of elimination of imaginaries for other ordered abelian groups, such as $\mathscr{H}_{i< \alpha} \Z \times \Q$ and $\sum_{i < \alpha} \Z \times \Q$, with $\alpha>1$ ordinal. Indeed, in a similar way we can prove
\begin{theor}
\label{NOTEIHAHNQ}
Let $\alpha>1$ be an ordinal. Then $\mathscr{H}_{i< \alpha} \Z \times \Q$ does not admit elimination of imaginaries in the pure language of ordered abelian groups.
\end{theor}
\begin{proof}
Let $\Omega=\mathscr{H}_{i< \alpha} \Z \times \Q$. Then $\Omega$ is the lexicographic product $\mathscr{H}_{i<\alpha+1}G_i$, where $G_i=\Z$ for every $i<\alpha$ and $G_{\alpha}=\Q$. As in the proof of Theorem \ref{NOTEI}, let $p$ be a prime, $a \in \Omega$ such that $a(0) \notin pG_0=p\Z$ and $X=[a]_{\equiv_p}$. Let $\bar{b}$ be a canonical parameter for $X$, $\bar{b}=(b_j)_{j< \mu}$ for some positive integer $\mu$.

Fix $h< \alpha$, and consider the automorphism  $\psi_h \colon \Omega \to \Omega$ defined by
\begin{equation*}
(\psi_h(f))(j)=
\begin{cases}
pf(j-1)+f(j) & \text{ if } j=h+1 \\
f(j) & \text{otherwise}
\end{cases}
\end{equation*}
for every $f \in \Omega$. Since $\psi_h(f) \equiv_p f$ for every $f \in \Omega$, in particular, $\psi_h \in \text{Stab}_{\Omega}(X)$. Then, by Fact \ref{prEI}, $\psi_h \in \text{Aut}(\Gamma / \bar{b})$, i.e. $\psi_h(b_j)=b_j$ for every $j<\mu$. Therefore, $p(b_j(h))=0$ for every $j<\mu$, and, so, $b_j(h)=0$ for every $j<\mu$. From the generality of $h<\alpha$, it follows that $b_j(h)=0$ for every $h<\alpha$ and every $j<\mu$, namely, $b_j \in \{0\}^{\alpha} \times \Q$ for every $j<\mu$.

Now consider the function $\phi \colon \Omega \to \Omega$ defined by
\begin{equation*}
(\phi(f))(j)=
\begin{cases}
2f(j) & \text{ if } j=\alpha \\
f(j) & \text{otherwise}
\end{cases}
\end{equation*}
for every $f \in \Omega$. Recall that $G_{\alpha}=\Q$. Then, $\phi$ is an order-preserving group automorphism and, trivially, $\phi \in \text{Stab}_{\Omega}(X)$. Therefore, by Fact \ref{prEI}, $\phi(b_j)=b_j$ for every $j<\mu$, and, so, $b_j(\alpha)=0$ for every $j<\mu$. It follows that $b_j=0$ for every $j<\mu$ and $X$ is $0$-definable. Hence, $\chi(X)=X$ for all automorphisms $\chi$ of $\Omega$. This gives a contradiction, since the automorphism $\chi \colon \Omega \to \Omega$ defined by
\begin{equation*}
(\chi(f))(j)=
\begin{cases}
f(0)+f(1) & \text{ if } j=1 \\
f(j) & \text{otherwise}
\end{cases} 
\end{equation*}
for every $f \in \Omega$, does not fix $X$.
\end{proof}
Note that the proof of Theorem \ref{NOTEIHAHNQ} also works for $\sum_{i < \alpha} \Z \times \Q$.

\begin{remark}
The above arguments can be used in order to prove the failure of elimination of imaginaries for many other ordered abelian groups, not only for lexicographic products. For example, let $p$ be a prime, $p \ne 2$, and consider $\Z_{(p)}$ as an ordered abelian group with the order induced from the usual order on $\Q$. Let $X=[1]_{\equiv_p}=1+p\Z_{(p)}$ be the $\equiv_p$-equivalence class of $1$. Define $\phi \colon \Z_{(p)} \to \Z_{(p)}$ by $\phi(w)=(p+1)w$. Then, $\phi$ is bijective and order-preserving, since $p+1>0$. Moreover, $\phi \in \text{Stab}(X)$, and $\text{Fix}(\phi)=\{0\}$. Therefore, if $\Z_{(p)}$ eliminates imaginaries (in $L_{oag}$), $X$ is $0$-definable and, in particular, $\text{Stab}(X)=\text{Aut}(\Z_{(p)})$. This is clearly false, since the automorphism $\psi \colon \Z_{(p)} \to \Z_{(p)}$ defined by $\psi(w)=2w$ does not fix $X$.
\end{remark}

\subsection{Weak elimination of imaginaries}
From Fact \ref{factpoizat}, one can deduce that the behaviour of the theories of $\mathscr{H}_{i< \alpha} \Z$ and $\mathscr{H}_{i< \alpha} \Z \times \Q$ in terms of "coding" imaginaries is even worse. Indeed, both $Th(\mathscr{H}_{i< \alpha}\Z)$ and $Th(\mathscr{H}_{i< \alpha} \Z \times \Q)$  do not have even weak elimination of imaginaries. Using a similar argument to that used in the proof of Theorem \ref{NOTEI}, we now provide a direct proof of the failure of weak elimination of imaginaries for the theory of $\mathscr{H}_{i< \alpha} \Z$.

\begin{theor}
Let $\alpha$ be an ordinal, $\alpha >1$. Then $T=Th(\mathscr{H}_{i<\alpha}\Z)$ does not have weak elimination of imaginaries in the pure language of ordered abelian groups. 
\end{theor}
\begin{proof}
Suppose for a contradiction that $T$ has weak elimination of imaginaries. For simplicity, consider $\Gamma=\sum_{i<\alpha}\Z$. As in the proof of Theorem \ref{NOTEI}, let $p$ be a prime and consider $X=[a]_{\equiv_p}$ the $\equiv_p$-equivalence class of $a \in \Gamma$ such that $a(0) \notin p\Z$. Then there exist a formula $\theta(x,\bar{w})$ and $B$ a finite set of $|\bar{w}|$-tuples such that $X=\theta(\Gamma,\bar{b})$ if and only if $\bar{b} \in B$. 

Let $\mu=|\bar{w}|$ and $\bar{b}=(b_0,\dots,b_{\mu}) \in B$. Then $\bar{b}$ is not $0$-definable, since $X$ is not $0$-definable. In particular, if $\alpha$ is a limit ordinal, then $b_j \ne 0$ for every $j \le \mu$, otherwise, if $\alpha=\beta+1$ for some $\beta < \alpha$, then $b_j \notin \Gamma_{\beta}$ for every $j \le \mu$. Fix $j \le \mu$ and consider $b_j$. Without loss of generality, we may assume $j=0$. Let $h=v(b_0)<\infty$. If $\alpha=\beta+1$, then $h<\beta$. In any case, since $h+1 < \alpha$, we can consider the function $\phi_{h} \colon \Gamma \to \Gamma$, defined as in (\ref{phih}). Therefore, since $\phi_h \in \text{Stab}_{\Gamma}(X)$, we have $X=\theta(\Gamma,\phi_h(\bar{b}))$ and $\phi_h(\bar{b}) \in B$. In particular, $\phi_h(\bar{b})=(\phi_h(b_0),\dots,\phi_h(b_{\mu}))$ and, by definition of $\phi_h$, we have $(\phi_h(b_0))(h+1)=p(b_0(h))+b_0(h+1)$. Then, from $b_0(h) \ne 0$, it follows that $(\phi_h(b_0))(h+1) \ne b_0(h+1)$ and $\phi_h(b_0) \ne b_0$. Therefore, $\phi_h(\bar{b}) \ne \bar{b}$. Now consider $\phi_h^{\delta}=\underbrace{\phi_h \circ \dots \circ \phi_h}_{\delta \text{ times}}$ for any natural number $\delta>0$. For every $\delta>0$, $\phi_h^{\delta} \in \text{Stab}_{\Gamma}(X)$ and, then, $\phi_h^{\delta}(\bar{b}) \in B$. Clearly, $(\phi_h^{\delta}(b_0))(h+1)=\delta p(b_0(h))+b_0(h+1)$. It follows that $\{\phi_h^{\delta}(\bar{b})\}_{\delta \in \mathbb{N}}$ is an infinite sequence of pairwise distinct elements of $\Gamma$. This gives a contradiction since $B$ is finite. 
\end{proof}

Similarly, we can provide a direct proof of the failure of  weak elimination of imaginaries for the theory of $\mathscr{H}_{i< \alpha} \Z \times \Q$.

\section{definable Skolem functions}
In this section, we prove that the theories of $\Lambda=\Z^n$ and $\Omega=\Z^n \times \Q$, with $n\ge1$, have definable Skolem functions once finitely many new constants are added to the language of ordered abelian groups. We recall that
\begin{definition}
Let $\mathfrak{M}$ be an $L$-structure. We say that $Th(\mathfrak{M})$ has definable Skolem functions if for every $L$-formula $\phi(\bar{x},y)$, there is an $L$-formula $\psi(\bar{x},y)$ such that
\begin{equation}
\label{Skolemfunc}
Th(\mathfrak{M}) \vdash \forall \bar{x} (\exists y \phi(\bar{x},y) \to (\exists ! y \psi(\bar{x},y) \land \forall y (\psi(\bar{x},y) \to \phi(\bar{x},y)))).
\end{equation} 
\end{definition}

First of all, note that the convex subgroup $C=\{0\}^n \times \Q$ of $\Z^n \times \Q$ is $0$-definable in $L_{oag}$ by the formula $\gamma(x)$ which says that all elements of $[0,\abs{x}]$ are divisible by $2$. We highlight
\begin{fact}
Let $G$ be an ordered abelian group. \\
(1) If $G \equiv \Z^n$, with $n \ge 1$, then $G$ has $n$ $0$-definable proper convex subgroups $G^{(0)}=\{0\} < G^{(1)}< \dots < G^{(n-1)} <G^{(n)}=G$ such that $G^{(i)} / G^{(i-1)} \equiv \Z$ for every $1 \le i \le n$. \\
(2) If $G \equiv \Z^n \times \Q$, with $n \ge 1$, then $G$ has $n+1$ $0$-definable proper convex subgroups $G^{(-1)}=\{0\} < G^{(0)}< \dots < G^{(n-1)} < G^{(n)}=G$ such that $G^{(0)} \equiv \Q$ and $G^{(i)} / G^{(i-1)} \equiv \Z$ for every $1 \le i \le n$.
\end{fact}

We rely on the following fact from \cite{weisp}. Consider the first order language $L_{Weis}=\Set{0,1^{(1)},1^{(2)},\dots,1^{(n)},+,-,<,(\equiv_m)_{m>0}}$, where $1^{(1)}$, $1^{(2)}$, $\dots, 1^{(n)}$ are new constant symbols. In $\Z^n$ we interpret $1^{(1)}, 1^{(2)}$, $\dots, 1^{(n)}$ as $(0,\dots,0,1), (0,\dots,0,1,0), \dots,$ \\ $(1,0,\dots,0)$, respectively. In $\Z^n \times \Q$ we interpret $1^{(1)}, 1^{(2)}$, $\dots, 1^{(n)}$ as $(0,\dots,0,1,0),$ $(0,\dots,0,1,0,0), \dots, (1,0,\dots,0)$, respectively. Namely, we expand $L_{oag}$ by the equivalence relations $\equiv_m$, for each positive integer $m$, and $n$ constants $1^{(i)}$, $1 \le i \le n$, for a representative of the smallest positive element in each discretely ordered quotient $G^{(i)} /G^{(i-1)}$, where either $G \equiv \Z^n$ or $G \equiv \Z^n \times \Q$ for some $n \ge 1$. In \cite{weisp}, Weisfenning proved that
\begin{fact}
\label{EQWeis}
Both $Th_{L_{Weis}}(\Z^n)$ and $Th_{L_{Weis}}(\Z^n \times \Q)$ admit elimination of quantifiers.
\end{fact}

Let $G$ be either a model of $Th_{L_{Weis}}(\Z^n)$ or a model of $Th_{L_{Weis}}(\Z^n \times \Q)$. In particular, every formula $\sigma(x,\bar{a})$ in $L_{Weis}$ with parameters $\bar{a} \subset G$, is equivalent modulo $G$ to a positive boolean combinations of formulas of the following forms
\begin{gather*}
(1) \quad kx \equiv_m t(\bar{a}), \text{or } kx \notequiv_m t(\bar{a}), \\
(2) \quad kx = t(\bar{a}), kx < t(\bar{a}), \text{or } t(\bar{a}) < kx,
\end{gather*}
where $k>0$ is a natural number. Let $t(\bar{a})=g \in G$. Since $[G : mG] < \infty$, every formula in (1) is equivalent to a finite disjuction of formulas of the form $kx \equiv_m g$. If such a formula defines a nonempty set, then there exists $h \in G$ such that $kx \equiv_m g$ is equivalent to $x \equiv_{m'} h$ for some $m'>0$. So we may assume that all formulas in (1) are of the form $x \equiv_m g$. Moreover, every formula in (2) defines a set which is a finite union of cosets of $2G$ intersected with intervals, see Theorem 12 and Theorem 15 in \cite{coset-minimal}. Therefore, every definable set in $G$ is a finite union of cosets of subgroups $mG$ intersected with intervals with endpoints in $G \cup \{\pm \infty\}$, for some positive integer $m$. Hence, $G$ is coset-minimal. We recall that
\begin{definition}
A totally ordered group (with possibly extra structure) is \emph{coset-minimal} if every definable set is a finite union of cosets of definable subgroups intersected with intervals. 
\end{definition}
In \cite{point}, it was proven that the groups elementarily equivalent to either $\Q$, or $\Z^n$, or $\Z^n \times \Q$, for some $n \ge 1$, are exactly the coset-minimal pure (modulo some constants) groups. Moreover, since any $\equiv_m$-equivalence class is $0$-definable, any such group provides an example of a quasi o-minimal structure (see \cite{quasi o-minimal}).

We show that the theories of $\Z^n$ and $\Z^n \times \Q$ have definable Skolem functions in $L_{Weis}$ and in a suitable language expanding $L_{Weis}$, respectively. 
In \cite{Skolemfunc}, using proof-theoretic arguments, Scowcroft identified the following sufficient condition for a model complete theory in order to have definable Skolem functions.

\begin{prop}
Let $L$ be a first order language with at least one constant symbol and $T$ be a model complete theory in $L$. Let $\Sigma$ be a set of \ $\forall \exists$-axioms for $T$ and $\Delta$ be the set of all quantifier free $L$-formulas $\delta(\bar{u},\bar{v})$ such that $\forall \bar{u} \exists \bar{v} \delta(\bar{u},\bar{v}) \in \Sigma$. Suppose for each $\delta(\bar{u},\bar{v}) \in \Delta$ there is an $L$-formula $\gamma_{\delta}(\bar{u},\bar{v})$ such that $T \vdash \forall \bar{u} \exists ! \bar{v} \ \gamma_{\delta}(\bar{u},\bar{v})$ and $T \vdash \forall \bar{u} \forall \bar{v} (\gamma_{\delta}(\bar{u},\bar{v}) \to \delta(\bar{u},\bar{v}))$. Then $T$ has definable Skolem functions.
\end{prop}

Note that by Fact \ref{EQWeis} the theories $Th_{L_{Weis}}(\Z^n)$ and $Th_{L_{Weis}}(\Z^n \times \Q)$ are model-complete and, hence, can be axiomatised by $\forall \exists$-sentences. Consider the following sets of sentences introduced in \cite[pp. 149-150]{tanaka}:
\begin{gather*}
\fleq{\begin{split}
	\Sigma_0 = & \forall x \forall y \forall z ((x+y)+z=x+(y+z)); \\
	& \forall x (x+0=x); \\
	& \forall x (x-x=0); \\
	& \forall x \forall y (x+y=y+x); \\
	& \forall x (\neg (x<x)); \\ 
	& \forall x \forall y \forall z (x<y<z \to x<z); \\
	& \forall x (x=0 \lor 0<x \lor x<0); \\
	& \forall x \forall y (0<x \land  0<y \to 0<x+y).
\end{split}} \\
\fleq{\begin{split}
\Sigma_1 = & \ 0< 21^{(i)}<1^{(i+1)} \text{ for each $i$ such that } 1 \le i \le n-1; \\
& \forall x(2x<1^{(i)} \lor 1^{(i)}<2x) \text{ for each $i$ such that } 1 \le i \le n; \\
& \forall x (2x<1^{(i)} \to mx < 1^{(i)}) \text{ for each $i$ such that } 1 \le i \le n \text{ and } m \ge 2; \\
& \forall x (x \equiv_m 0 \leftrightarrow \exists y \exists z(-1^{(1)}<2z<1^{(1)} \land x=my+z)) \text{ for each } m>0; \\
& \forall x (\underset{0 \le q_1,\dots,q_n \le m-1}{\bigvee}(x \equiv_m q_11^{(1)}+\dots+q_n1^{(n)})) \text{ for each } m>1; \\
& \forall x (-1^{(1)}<2x<1^{(1)} \to \exists y(x=my)) \text{ for each } m>1.
\end{split}} \\
\fleq{\Sigma_2 = \forall x (\neg (0<x<1^{(1)})).} \\
\fleq{\Sigma_3 = \exists x (0<x<1^{(1)}).}
\end{gather*}
It was shown in \cite{tanaka} that the theories $Th_{L_{Weis}}(\Z^n)$ and $Th_{L_{Weis}}(\Z^n \times \Q)$ are axiomatized by $\Sigma_0 \cup \Sigma_1 \cup \Sigma_2$ and $\Sigma_0 \cup \Sigma_1 \cup \Sigma_3$, respectively. Therefore we are able to use Scowcroft's criterion for the existence of definable Skolem functions in model complete theories. It follows easily that
\begin{theor}
\label{defSkolfunc1}
The theory $Th_{L_{Weis}}(\Z^n)$ has definable Skolem functions.
\end{theor}
\begin{proof}
We just need to show the existence of definable Skolem functions for the non-universal axioms:
\begin{gather*}
\forall x (x \equiv_m 0 \leftrightarrow \exists y \exists z(-1^{(1)}<2z<1^{(1)} \land x=my+z)) \text{ for each } m>0, \text{ and} \\
\forall x (-1^{(1)}<2x<1^{(1)} \to \exists y(x=my)) \text{ for each } m>1.
\end{gather*}
It is sufficient to note that, in any model $G$ of $Th_{L_{Weis}}(\Z^n)$, if $g \in G$ is such that $g=mh$ for some $h \in G$, then $h$ is unique. Indeed, let $g,h,k \in G$ be such that $g=mh+k$ and $-1^{(1)}<2k<1^{(1)}$. Trivially, by the interpretation of $1^{(1)}$, the inequalities $-1^{(1)}<2k<1^{(1)}$ imply $k=0$. Hence, $k$ is unique and, so also $h$ in $g=mh$ is unique.
\end{proof}

In order to prove the same result for $\Z^n \times \Q$, for any $n \ge 1$, fix an element $c$ in $C=\{0\}^n \times \Q$ such that $c>0$. Then the quantifier elimination in $Th_{L_{Weis}}(\Z^n \times \Q)$ is not affected by adding $c$ as a new  constant to $L_{Weis}$. Let $L_{Weis}(c)$ denote $L_{Weis} \cup \{c\}$. Now an axiomatization of $Th_{L_{Weis}(c)}(\Z^n \times \Q)$ is given by $\Sigma_0 \cup \Sigma_1 \cup \Sigma_3$ and the following axiom: 
\begin{equation*}
0<2c<1^{(1)}.
\end{equation*} 

\begin{theor}
\label{defSkolfunc2}
The theory $Th_{L_{Weis}(c)}(\Z^n \times \Q)$ has definable Skolem functions.
\end{theor}
\begin{proof}
We just need to show the existence of definable Skolem functions for the non-universal axioms:
\begin{gather*}
\forall x (x \equiv_m 0 \leftrightarrow \exists y \exists z(-1^{(1)}<2z<1^{(1)} \land x=my+z)) \text{ for each } m>0, \text{ and} \\
\forall x (-1^{(1)}<2x<1^{(1)} \to \exists y(x=my)) \text{ for each } m>1, \text{ and} \\
\exists x (0<x<1^{(1)}).
\end{gather*}
Note that $0<c<1^{(1)}$. Then, as in the proof of Theorem \ref{defSkolfunc1}, it suffices to note that, in any model $G$ of $Th_{L_{Weis}}(\Z^n \times \Q)$, if $g \in G$ is such that $g=mh$ for some $h \in G$, then $h$ is unique. Indeed, let $g,h,k \in G$ be such that $g=mh+k$ and $-1^{(1)}<2k<1^{(1)}$. In particular, there exists $k' \in G$ such that $k=mk'$, and $g=m(h+k')$. Therefore, we may assume $k=0$.  
\end{proof}

By the characterization of dp-minimal ordered groups in \cite{dp-minimal}, for any $n \ge 1$, the finite lexicographic products $\Z^n$ and	$\Z^n \times \Q$ are dp-minimal. The author has recently learned that Vicaria \cite{vicaria} has indentified a suitable many-sorted language in which the class of dp-minimal ordered abelian groups eliminates imaginaries. We aimed at identifying a single-sorted language that could suffice for eliminating imaginaries for the theory of the ordered groups $\Z^n$, and $\Z^n\times\Q$, for any $n \ge 1$. The languages $L_{Weis}$ and $L_{Weis}(c)$ seemed to be promising for this goal, since in these languages we can eliminate the $\equiv_p$-equivalence classes and,
by Theorem \ref{defSkolfunc1} and Theorem \ref{defSkolfunc2}, there are definable Skolem functions. Indeed, it is well known that, provided that one can uniformly associate a canonical parameter to every unary definable set, the existence of definable Skolem functions is a sufficient condition for uniform elimination of imaginaries (see \cite[Lemma 4.4.3]{hodges}).

The following example pointed out to the author by M. Hils shows that unfortunately this is not plausible, and a many-sorted language seems unavoidable. Consider the lexicographic product $G=\Z \times \mathbb{R} \times \Z$ in $L_{Weis}$ (or in some expansion $L$ of $L_{Weis}$ by adding new constants) and suppose $G$ admits elimination of imaginaries. By Fact \ref{EQWeis}, we have that every infinite definable set $X \subseteq G^m$ is uncountable. Then, in particular, the set of canonical parameters of cosets $a + \{0\} \times \mathbb{R} \times \Z$, with $a \in G$, is uncountable. This is clearly false, since $\frac{\Z \times \mathbb{R} \times \Z}{\{0\} \times \mathbb{R} \times \Z} \cong \Z$. The contradiction follows from the existence of a canonical parameter for $a + \{0\} \times \mathbb{R} \times \Z$, and so $G \models Th_{L_{Weis}}(\Z^2)$ does not admit elimination of imaginaries. Similar arguments can be used to show that the theory of $\Z^n$, for any $n>1$, and the theory of $\Z^n \times \Q$, for any $n \ge 1$, do not admit elimination of imaginaries in any expansion $L$ of $L_{oag}$ by adding new constants.

\begin{remark}
It is still unsolved the problem of eliminating imaginaries for the theories of $\mathscr{H}_{i < \alpha} \Z$ and $\mathscr{H}_{i < \alpha} \Z \times \Q$ with $\alpha$ any ordinal, since these groups do not belong to the class of ordered abelian groups considered in \cite{vicaria}.
\end{remark}

\subsection*{\textit{Acknowledgements}} I would like to thank Immanuel Halupczok for suggesting this topic. I owe a special thank to Lorna Gregory and Angus Macintyre for many fruitful discussions and useful remarks. I would also like to express my gratitude to my supervisor Paola D'Aquino for her constant guidance.

\end{document}